\documentclass{article}%
\usepackage{amssymb}
\usepackage{amsfonts}
\usepackage{amsmath}
\usepackage{graphicx}%
\providecommand{\U}[1]{\protect\rule{.1in}{.1in}}
\newtheorem{theorem}{Theorem}

\newtheorem{lemma}[theorem]{Lemma}

\newtheorem{proposition}{Proposition}

\newenvironment{proof}[1][Proof]{\noindent\textbf{#1.} }{\ \rule{0.5em}{0.5em}}
\begin{document}

\title{Concomitants and majorization bounds for bivariate distribution function}
\author{Ismihan Bairamov\\Department of Mathematics, Izmir University of Economics \\Izmir, Turkey. \ E-mail: ismihan.bayramoglu@ieu.edu.tr }
\maketitle

\begin{abstract}
Let ($X,Y)$ be a random vector with distribution function $F(x,y),$ and
$(X_{1},Y_{1}),(X_{2},Y_{2}),...,(X_{n},Y_{n})$ are independent copies of
($X,Y).$ Let $X_{i:n}$ be the $i$th order statistics constructed from the
sample $X_{1},X_{2},...,X_{n}$ of the first coordinate of the bivariate sample
and $Y_{[i:n]}$ be the concomitant of $X_{i:n}.$ Denote $F_{i:n}%
(x,y)=P\{X_{i:n}\leq x,Y_{[i:n]}\leq y\}.$ \ Using majorization theory we
write upper and lower bounds for $F$ expressed in terms of mixtures of joint
distributions of order statistics and their concomitants, i.e. $%
{\displaystyle\sum\limits_{i=1}^{n}}
p_{i}F_{i:n}(x,y)$ and $%
{\displaystyle\sum\limits_{i=1}^{n}}
p_{i}F_{n-i+1:n}(x,y).$ \ It is shown that these bounds converge to $F$ \ for
a particular sequence $(p_{1}(m),p_{2}(m),...,p_{n}(m)),m=1,2,..$ as
$m\rightarrow\infty.$

\end{abstract}

\section{Introduction}

Let ($X_{1},Y_{1}),(X_{2},Y_{2}),...,(X_{n},Y_{n})$ be independent and
identically distributed (iid) random vectors with joint distribution function
(cdf) $F(x,y)$ and $X_{1:n}\leq X_{2:n}\leq\cdots\leq X_{n:n}$ be the order
statistics of the sample of first coordinate $X_{1},X_{2},...,X_{n}.$ Denote
the $Y$ -\ variate associated with $X_{i:n}$ by $Y_{[i:n]},$ $i=1,2,..,n$,
i.e. $Y_{[i:n]}=Y_{k}$ iff $X_{i:n}=X_{k}.$ The random variables
$Y_{[1:n]},Y_{[2:n]},...,Y_{[n:n]}$ are called concomitants of order
statistics $X_{1:n},X_{2:n},...,X_{n:n}.$ The theory of order statistics is
well documented in David (1981), David and Nagaraja (2003), Arnold et al.
(1992). \ The concomitants of order statistics are described in David (1973),
Bhattacharya (1974), \ David and Galambos \ (1974), David and Nagaraja (1998),
Wang (2008). Denote by $F_{i:n}(x,y)$ the joint distribution of order
statistic $X_{i:n}$ and its concomitant $Y_{[i:n]}.$

Let $X_{1},X_{2},...,X_{n}$ be a univariate sample with cdf $F(x)$ and
$F_{i:n}(x)=P\{X_{i:n}\leq x\},$ where $X_{i:n}$ is the $i$th order statistic
of this sample. Recently, Bairamov (2011) considered mixtures of distribution
functions of order statistics $K_{n}(x):=%
{\displaystyle\sum\limits_{i=1}^{n}}
p_{i}F_{i:n}(x)$ and $H_{n}(x):=%
{\displaystyle\sum\limits_{i=1}^{n}}
p_{i}F_{n-i+1:n}(x)$ and using inequalities of majorization theory showed that
for a particular choice of $p_{i}$'s, $\ H_{n}(x)\leq F(x)\leq K_{n}(x)$ for
all $x\in%
\mathbb{R}
$ and the $L_{2}$ distance between $H_{n}(x)$ and $K_{n}(x)$ can be made
sufficiently small. In other words there exists a sequence $(p_{1}%
,p_{2},...,p_{n})=($ $p_{1}(m),p_{2}(m),...,p_{n}(m))$ such that for this
sequence $H_{n}(x)$ and $K_{n}(x)$ converge to $F(x)$ as $m\rightarrow\infty$
with rate $o(1/x^{1+\alpha}),$ $0<\alpha<1$ and the $L_{1}$ distance between
$H_{n}(x)$ and $K_{n}(x)$ can be made as small as required. $\ $

In this paper we extend the results presented in Bairamov (2011) to the
bivariate case. We consider mixtures of joint distribution functions of order
statistics and concomitants $K_{n}(x,y):=%
{\displaystyle\sum\limits_{i=1}^{n}}
p_{i}F_{i:n}(x,y)$ and $H_{n}(x,y):=%
{\displaystyle\sum\limits_{i=1}^{n}}
p_{i}F_{n-i+1:n}(x,y).$ \ Using majorization inequalities it is shown that for
a particular sequence $(p_{1},p_{2},...,p_{n})$ $=(p_{1}(m),p_{2}%
(m),...,p_{n}(m)),$ $m=1,2,...$, $\ $it is true that $H_{n}(x,y)\leq
F(x,y)\leq K_{n}(x,y)$ for all $(x,y)\in%
\mathbb{R}
$ and the distance between $H_{n}(x,y)$ and $K_{n}(x,y)$ goes to zero as
$m\rightarrow\infty.$ \ 

\section{Auxiliary Results}

Let $\ \mathbf{a=(}a_{1},a_{2},...,a_{n})\in%
\mathbb{R}
^{n}$ , $\mathbf{b}=(b_{1},b_{2},...,b_{n})\in%
\mathbb{R}
^{n}$ and $a_{[1]}\geq a_{[2]}\geq\cdots\geq a_{[n]}$ denote the components of
$\mathbf{a}$ in decreasing order. The vector $\mathbf{a}$ is said to be
majorized by the vector $\mathbf{b}$ and $\text{denoted by }\mathbf{a}%
\prec\mathbf{b},$ if
\[
\sum_{i=1}^{k}a_{[i]}\leq\sum_{i=1}^{k}b_{[i]}\text{ for }k=1,2,\cdots,n-1
\]
and%
\[
\sum_{i=1}^{n}a_{[i]}=\sum_{i=1}^{n}b_{[i].}%
\]
The details of the theory of majorization can be found in Marshall et al.
(2011). The following two theorems are important for our study.

\begin{proposition}
\label{Proposition 1} Denote $D=\{(x_{1},x_{2},...,x_{n}):x_{1}\geq x_{2}%
\geq\cdots\geq x_{n}\},$ $\mathbf{a=(}a_{1},a_{2},...,a_{n}),$ $\mathbf{b}$
$=(b_{1},b_{2},...,b_{n}).$ The inequality
\[
\sum_{i=1}^{n}a_{i}x_{i}\leq\sum_{i=1}^{n}b_{i}x_{i}%
\]
holds for all $(x_{1},x_{2},...,x_{n})\in D$ if and only if $\mathbf{a}%
\prec\mathbf{b}$ in $D.$(Marshal et al. 2011, page 160).
\end{proposition}

\begin{proposition}
\bigskip\label{Proposition 2}The inequality
\[
\sum_{i=1}^{n}a_{i}x_{i}\leq\sum_{i=1}^{n}b_{i}x_{i}%
\]
holds whenever $x_{1}\leq x_{2}\leq\cdots\leq x_{n}$ if and only if
\begin{align*}%
{\displaystyle\sum\limits_{i=1}^{k}}
a_{i}  &  \geq%
{\displaystyle\sum\limits_{i=1}^{k}}
b_{i},\text{ }k=1,2,...,n-1\\%
{\displaystyle\sum\limits_{i=1}^{n}}
a_{i}  &  =%
{\displaystyle\sum\limits_{i=1}^{n}}
b_{i}.
\end{align*}

\end{proposition}

(Marshall et al. 2011, page 639).

\section{Main results in bivariate case}

Let $(X,Y)$ be absolutely continuous random vector with joint cdf $F(x,y)$ and
probability density function (pdf) $f(x,y)$. Let $(X_{1},Y_{1}),(X_{2}%
,Y_{2}),...,(X_{n},Y_{n})$ be independent copies of $(X,Y).$ Let $X_{r:n}$ be
the $r$th order statistic and $Y_{[r:n]}$ be its concomitant, i.e.
$Y_{[r:n]}=Y_{i}$ iff $X_{r:n}=X_{i}.$ \ The joint distribution of $X_{r:n}$
and $Y_{[r:n]}$ can be easily derived and it is%
\begin{align*}
F_{r:n}(x,y)  &  =B_{n}%
{\displaystyle\int\limits_{-\infty}^{x}}
F_{X}^{r-1}(u)(1-F_{X}(u))^{n-r}F(du,y)du\\
&  =B_{n}%
{\displaystyle\int\limits_{-\infty}^{x}}
F_{X}^{r-1}(u)(1-F_{X}(u))^{n-r}\left(
{\displaystyle\int\limits_{-\infty}^{y}}
f(u,v)dv\right)  du,
\end{align*}
where
\[
F(du,y)=\frac{\partial}{\partial u}F(u,v)=%
{\displaystyle\int\limits_{-\infty}^{y}}
f(u,v)dv,\text{ and }B_{n}=n\binom{n-1}{r-1}.
\]
It \ is easy to check that
\begin{equation}
\frac{1}{n}%
{\displaystyle\sum\limits_{r=1}^{n}}
F_{r:n}(x,y)=F(x,y). \label{d1}%
\end{equation}

\begin{lemma}
\bigskip\label{Lemma 1}$F_{r+1:n}(x,y)\leq F_{r:n}(x,y)$, $r=1,2,...,n-1,$ for
all $(x,y)\in%
\mathbb{R}
^{2}.$
\end{lemma}

\begin{proof}
We have
\begin{align}
&  F_{r+1:n}(x,y)-F_{r:n}(x,y)\nonumber\\
&  =n\binom{n-1}{r}%
{\displaystyle\int\limits_{-\infty}^{x}}
F_{X}^{r}(u)(1-F_{X}(u))^{n-r-1}F(du,y)du\nonumber\\
&  -n\binom{n-1}{r-1}%
{\displaystyle\int\limits_{-\infty}^{x}}
F_{X}^{r-1}(u)(1-F_{X}(u))^{n-r}F(du,y)du\nonumber\\
&  =%
{\displaystyle\int\limits_{-\infty}^{x}}
\left[  n\binom{n-1}{r}F_{X}^{r}(u)(1-F_{X}(u))^{n-r-1}-\right. \nonumber\\
&  \left.  n\binom{n-1}{r-1}F_{X}^{r-1}(u)(1-F_{X}(u))^{n-r}\right]
F(du,y)du\nonumber\\
&  =%
{\displaystyle\int\limits_{0}^{F_{X}(x)}}
\left[  n\binom{n-1}{r}t^{r}(1-t)^{n-r-1}\right. \label{d2}\\
&  \left.  -n\binom{n-1}{r-1}t^{r-1}(1-t)^{n-r}\right]  F(dF_{X}%
^{-1}(t),y)dt,\nonumber
\end{align}
where
\[
F(dF_{X}^{-1}(t),y)=%
{\displaystyle\int\limits_{-\infty}^{y}}
f(F_{X}^{-1}(t),v)dv.
\]
Since
\[
h(t):=n\binom{n-1}{r}t^{r}(1-t)^{n-r-1}-n\binom{n-1}{r-1}t^{r-1}%
(1-t)^{n-r}\text{ and }g(t):=F(dF_{X}^{-1}(t),y)
\]
are both bounded integrable functions in $t\in\lbrack0,F_{X}^{-1}(x)],$ for
all $x,y\in%
\mathbb{R}
$ and $g(x)$ is one sided function in this interval, then by the first mean
value theorem for integral (see Gradshteyn and Ryzhik (2007 ), 12.1, page
1053)
\[%
{\displaystyle\int\limits_{0}^{F_{X}(x)}}
h(t)g(t)dt=g(\xi)%
{\displaystyle\int\limits_{0}^{F_{X}(x)}}
h(t)dt,
\]
where $0\leq\xi_{x}\leq F_{X}(x),$ $x\in%
\mathbb{R}
,$ $g(\xi_{x})\geq0.$ The last equality together with (\ref{d2}) leads to%
\begin{align*}
&  F_{r+1:n}(x,y)-F_{r:n}(x,y)\\
&  =g(\xi_{x})%
{\displaystyle\int\limits_{0}^{F_{X}(x)}}
\left[  n\binom{n-1}{r}t^{r}(1-t)^{n-r-1}-n\binom{n-1}{r-1}t^{r-1}%
(1-t)^{n-r}\right]  dt\\
&  =g(\xi_{x})\left\{
{\displaystyle\int\limits_{0}^{F_{X}(x)}}
n\binom{n-1}{r}t^{r}(1-t)^{n-r-1}-%
{\displaystyle\int\limits_{0}^{F_{X}(x)}}
\frac{n!}{(r-1)!(n-r)!}t^{r-1}(1-t)^{n-r}dt\right\} \\
&  =g(\xi_{x})[P\{X_{r+1:n}\leq x\}-P\{X_{r:n}\leq x\}]\leq0.
\end{align*}

\end{proof}

Denote
\[
D_{+}^{1}=\{(x_{1},x_{2},...,x_{n}):x_{i}\geq0,i=1,2,...,n;\text{ }x_{1}\geq
x_{2}\geq\cdots\geq x_{n},\sum_{i=1}^{n}x_{i}=1\}.
\]

\begin{lemma}
\label{Lemma 2}Let $(p_{1},p_{2},...,p_{n})\in D_{+}^{1}$ . Then
\begin{align}
H_{n}(x,y)  & \equiv%
{\displaystyle\sum\limits_{i=1}^{n}}
p_{i}F_{n-i+1:n}(x,y)\leq F(x,y)\leq%
{\displaystyle\sum\limits_{i=1}^{n}}
p_{i}F_{i:n}(x,y)\equiv K_{n}(x,y)\text{ }\label{b5}\\
\text{for all (}x,y)  & \in%
\mathbb{R}
^{2}\nonumber
\end{align}
and the equality holds if and only if $(p_{1},p_{2},\cdots,p_{n})=(\frac{1}%
{n},\frac{1}{n},...,\frac{1}{n}).$ Furthermore, if \ $\mathbf{p}=(p_{1}%
,p_{2},...,p_{n})\in D_{+}^{1},$ $\mathbf{q=}(q_{1},q_{2},...,q_{n})\in
D_{+}^{1}$ and $\mathbf{p}\prec\mathbf{q,}$ then
\begin{align}%
{\displaystyle\sum\limits_{i=1}^{n}}
q_{i}F_{n-i+1:n}(x,y)  & \leq%
{\displaystyle\sum\limits_{i=1}^{n}}
p_{i}F_{n-i+1:n}(x,y)\leq F(x,y)\leq%
{\displaystyle\sum\limits_{i=1}^{n}}
p_{i}F_{i:n}(x,y)\label{b6}\\
& \leq%
{\displaystyle\sum\limits_{i=1}^{n}}
q_{i}F_{i:n}(x,y).\nonumber
\end{align}

\end{lemma}

\begin{proof}
By \ Lemma 1 $F_{1:n}(x,y)$ $\geq F_{2:n}(x,y)\geq\cdots\geq F_{n:n}(x,y)$ for
all $x\in%
\mathbb{R}
,$ and $(\frac{1}{n},\frac{1}{n},...,\frac{1}{n})\prec(p_{1},p_{2}%
,...,p_{n}),$ the right hand side of the inequality (\ref{b5}) follows from
the Proposition 1 and left hand side follows from Proposition 2.
\end{proof}

\begin{theorem}
Let $p_{i}(m)=\frac{m+n-i+1}{a_{n}(m)},$ $i=1,2,...,n;$ $\ \ m\in
\{0,1,2,...\},$ where $a_{n}(m)=nm+\frac{n(n+1)}{2}.$ Then
\begin{align}
H_{n}^{(m)}(x,y)  & \equiv%
{\displaystyle\sum\limits_{i=1}^{n}}
p_{i}(m)F_{n-i+1:n}(x,y)\leq F(x,y)\label{b7}\\
& \leq%
{\displaystyle\sum\limits_{i=1}^{n}}
p_{i}(m)F_{i:n}(x,y)\equiv K_{n}^{(m)}(x,y)\text{ for all }(x,y)\in%
\mathbb{R}
^{2}\nonumber
\end{align}
and%
\begin{align}
\underset{m\rightarrow\infty}{\ \lim}%
{\displaystyle\sum\limits_{i=1}^{n}}
p_{i}(m)F_{n-i+1:n}(x,y)  & =\underset{m\rightarrow\infty}{\ \lim}%
{\displaystyle\sum\limits_{i=1}^{n}}
p_{i}(m)F_{i:n}(x,y)\label{b8}\\
& =F(x,y)\text{ for all }(x,y)\in%
\mathbb{R}
^{2}.\nonumber
\end{align}
Furthermore,
\begin{equation}%
{\displaystyle\int\limits_{-\infty}^{\infty}}
{\displaystyle\int\limits_{-\infty}^{\infty}}
\left\vert K_{n}^{(m)}(x,y)-H_{n}^{(m)}(x,y)\right\vert dxdy=o(\frac
{1}{m^{1-\alpha}}),\text{ }0<\alpha<1.\label{b9}%
\end{equation}

\end{theorem}

\begin{proof}
\bigskip Consider $p_{i}(m)=\frac{m+n-i+1}{a_{n}(m)},$ $i=1,2,...,n;$
$\ \ m\in\{0,1,2,...\},$ where $a_{n}(m)=nm+\frac{n(n+1)}{2}.$ It is clear
that $p_{1}(m)\geq p_{2}(m)\geq\cdots\geq p_{n}(m)$ and $%
{\displaystyle\sum\limits_{i=1}^{n}}
p_{i}(m)=1.$ Since $(\frac{1}{n},\frac{1}{n},...,\frac{1}{n})\prec
(p_{1}(m),p_{2}(m),...,p_{n}(m))$ then from Lemma 1 we have
\begin{equation}%
{\displaystyle\sum\limits_{i=1}^{n}}
p_{i}(m)F_{n-i+1:n}(x,y)\leq F(x,y)\leq%
{\displaystyle\sum\limits_{i=1}^{n}}
p_{i}(m)F_{i:n}(x,y). \label{3b}%
\end{equation}
Since
\[
\underset{m\rightarrow\infty}{\lim}p_{i}(m)=\underset{m\rightarrow\infty}%
{\lim}\frac{m+i}{nm+\frac{n(n+1)}{2}}=\frac{1}{n},\text{ }i=1,2,...,n,
\]
(\ref{b8}) follows. To prove (\ref{b9}) consider the $L_{1}$ distance between
$K_{n}^{(m)}(x,y)$ and $H_{n}^{(m)}(x,y).$ We have
\begin{align*}
\Delta_{m}  &  \equiv%
{\displaystyle\int\limits_{-\infty}^{\infty}}
{\displaystyle\int\limits_{-\infty}^{\infty}}
\left\vert K_{n}^{(m)}(x,y)-H_{n}^{(m)}(x,y)\right\vert dxdy\\
&  =%
{\displaystyle\int\limits_{-\infty}^{\infty}}
{\displaystyle\int\limits_{-\infty}^{\infty}}
\left\vert
{\displaystyle\sum\limits_{i=1}^{n}}
p_{i}(m)F_{i:n}(x,y)-%
{\displaystyle\sum\limits_{i=1}^{n}}
p_{i}(m)F_{n-i+1:n}(x,y)\right\vert dxdy\\
&  =%
{\displaystyle\int\limits_{-\infty}^{\infty}}
{\displaystyle\int\limits_{-\infty}^{\infty}}
\left\vert
{\displaystyle\sum\limits_{i=1}^{n}}
p_{i}(m)F_{i:n}(x,y)-F(x,y)+F(x,y)-%
{\displaystyle\sum\limits_{i=1}^{n}}
p_{i}(m)F_{n-i+1:n}(x,y)\right\vert dxdy\\
&  =%
{\displaystyle\int\limits_{-\infty}^{\infty}}
{\displaystyle\int\limits_{-\infty}^{\infty}}
\left\vert
{\displaystyle\sum\limits_{i=1}^{n}}
(p_{i}(m)-\frac{1}{n})F_{i:n}(x,y)+(\frac{1}{n}-p_{i}(m))F_{n-i+1:n}%
(x,y)\right\vert dxdy\\
&  \leq%
{\displaystyle\sum\limits_{i=1}^{n}}
\left\vert p_{i}(m)-\frac{1}{n}\right\vert
{\displaystyle\int\limits_{-\infty}^{\infty}}
{\displaystyle\int\limits_{-\infty}^{\infty}}
\left\vert F_{i:n}(x,y)-F_{n-i+1:n}(x,y)\right\vert dxdy\\
&  \leq(p_{1}(m)-\frac{1}{n})c_{n}=\frac{\frac{1}{m}\frac{n(n-1)}{2}}%
{n^{2}+\frac{n^{2}(n+1)}{2}\frac{1}{m}}c_{n},
\end{align*}
where $c_{n}=%
{\displaystyle\sum\limits_{i=1}^{n}}
{\displaystyle\int\limits_{-\infty}^{\infty}}
{\displaystyle\int\limits_{-\infty}^{\infty}}
\left\vert F_{i:n}(x,y)-F_{n-i+1:n}(x,y)\right\vert dxdy.$
\end{proof}

\end{document}